%% file: CDC21_Hahn.tex
\documentclass[10pt, conference]{ieeeconf}      

\IEEEoverridecommandlockouts                              

\usepackage{amsthm}
\newtheorem{thm}{Theorem}
\newtheorem{lem}[thm]{Lemma}
\newtheorem{prop}[thm]{Proposition}

\newtheorem{defn}[thm]{Definition}
\theoremstyle{definition}
\newtheorem{rem}[thm]{Remark}
\newtheorem{assum}[thm]{Assumption}
\input{./Inputs/MyPackages.tex}
\input{./Inputs/uks_Commands.tex}
\usepackage{yfonts} 

\usepackage[ngerman,english]{babel}
\useshorthands{"}											
\addto\extrasenglish{\languageshorthands{ngerman}}

\title{\LARGE \bf	Uncertain AoI in Stochastic Optimal Control of Networked and Constrained LTI Systems}

\author{Jannik Hahn$^{1}$ and Olaf Stursberg$^{1}$
\thanks{$^{1}$Control and System Theory, Dept. of Electrical Engineering and Computer Science, University of Kassel, 34121 Kassel, Germany. Email: \{jhahn,stursberg\}@uni-kassel.de.}%
}

\begin{document}
\maketitle
\thispagestyle{empty}
\pagestyle{empty}



\input{./Inputs/Abstract.tex}


\section{INTRODUCTION}
\input{./Inputs/Introduction.tex}
\section{PRELIMINARIES}
\input{./Inputs/Preliminaries.tex}
\section{NETWORKED SYSTEM AND PROBLEM DEFINITION}
\input{./Inputs/ProbDef.tex}

\section{AGE OF INFORMATION}
\input{./Inputs/AgeOfInformation.tex}
\section{CONTROLLER SYNTHESIS}
\input{./Inputs/ControllerSynthesis.tex}
\section{Numerical Example}
\input{./Inputs/Simulation.tex}

\section{CONCLUSIONS}
\input{./Inputs/Conclusion.tex}




\appendix
                           	
\input{./Inputs/Appendix.tex}

\section*{Acknowledgment}

\input{./Inputs/Acknowledgments.tex}



\bibliographystyle{IEEEtran}
\bibliography{IEEEabrv,Bibliography.bib}

\end{document}

%% file: Inputs/MyPackages.tex
\usepackage{amsmath}
\usepackage{amssymb}
\usepackage{amsfonts}
\usepackage{graphicx}
\graphicspath{{./Figures/}}
\usepackage{color}
\usepackage{delarray}
\usepackage{subcaption}
\usepackage{psfrag}
\usepackage{ifthen}
\usepackage{mathtools}
\mathtoolsset{showonlyrefs,showmanualtags}		
\usepackage[scr=zapfc, scrscaled=1.25]{mathalfa}
\usepackage[binary-units=true]{siunitx}
\DeclareSIUnit\permille{\text{\textperthousand}}

\usepackage{tabularx}

\usepackage{stfloats}

\DeclareFontFamily{U}{mathc}{}
\DeclareFontShape{U}{mathc}{m}{it}%
{<->s*[1] mathc10}{}

\DeclareMathAlphabet{\mathfrak}{U}{mathc}{m}{it}

%% file: Inputs/uks_Commands.tex
\newcommand{\s}[2][]{^{(#2)^{#1}}}
\newcommand{\hs}[2][]{^{[#2]^{#1}}}

\newcommand{\setdef}[2]{\left\{ #1 \left\vert \begin{array}{l}#2\end{array} \hspace*{-1ex} \right.\right\}}

\newcommand{\norm}[1]{\left\lVert#1\right\rVert}

\newenvironment{salign}{\thinmuskip=1mu\medmuskip=1mu\thickmuskip=1mu\align}{\endalign}
\newenvironment{salign*}{\thinmuskip=1mu\medmuskip=2mu\thickmuskip=3mu \csname align*\endcsname}{\csname endalign*\endcsname}
\newenvironment{bm}{\setlength\arraycolsep{2pt}\bmatrix}{\endbmatrix}
\newenvironment{sbm}{\small \bm}{\endbm}

\newcommand{\figref}[1]{Fig.~\ref{#1}}

\DeclareMathOperator{\blkdiag}{blkdiag}
\DeclareMathOperator{\diag}{diag}
\DeclareMathOperator{\trace}{trace}
\DeclareMathOperator{\expected}{\mathbb{E}}
\newcommand{\E}[2][]{\expected_{#1}\left[#2\right]}

\DeclareMathOperator{\cov}{cov}

\DeclareMathOperator{\row}{row}
\let\P\relax\DeclareMathOperator{\P}{\textbf{P}}
\newcommand{\ellipsoid}[2]{\varepsilon\left(#1,#2\right)}
\newcommand{\NV}[2]{\mathcal{N}\left(#1,#2\right)}
\let\P\relax\DeclareMathOperator{\P}{\mathbb{P}}
\newcommand{\pr}[2][]{\P_{#1}\left(#2\right)}
\newcommand{\prr}[2][]{\P_{#1}\big(#2\big)}
\newcommand{\B}[1]{\mathcal{B}\left(#1\right)}

\newcommand{\kpl}[1]{\ifthenelse{\equal{#1}{0}}{_{0}}{_{#1}}}
\newcommand{\kPl}[2]{_{k\ifthenelse{\equal{#1}{0}}{}{+#1},\ifthenelse{\equal{#2}{0}}{}{+#2}}}
\newcommand{\kps}[1]{_{#1}}
\newcommand{\kPs}[2]{_{#1,#2}}

\newcommand{\kp}[1]{\kpl{k+#1}}

\newcommand{\K}{K}
\newcommand{\xf}{\textbf{x}}
\newcommand{\Kf}{\textbf{\K}}
\newcommand{\uf}{\textbf{u}}

\newcommand{\wf}{\textbf{w}}

\newcommand{\Pf}{\textbf{P}}
\newcommand{\Af}{\textbf{A}}
\newcommand{\Bf}{\textbf{B}}
\newcommand{\Cf}{\textbf{C}}
\newcommand{\Ef}{\textbf{E}}

\newcommand{\Vf}{\textbf{V}}
\newcommand{\Mf}{\textbf{M}}

\newcommand{\Sf}{\textbf{S}}
\newcommand{\du}{{\delta_u}}
\newcommand{\dx}{{\delta_x}}
\newcommand{\predH}{\mathbb{N}_H}
\newcommand{\p}{p}
\newcommand{\st}{\text{s.t.:}\quad}

\usepackage{dsfont}
\newcommand{\ident}{\mathds{1}}

\newcommand{\Sc}{\mathcal{S}}
\newcommand{\Xc}{\mathcal{X}}

\newcommand{\Wc}{\mathcal{W}}

\newcommand{\Am}{\mathcal{A}}
\newcommand{\Em}{\mathscr{E}}

\newcommand{\Wf}{\textbf{W}}

\newcommand{\Ls}{\mathcal{L}\mathfrak{s}}
\newcommand{\Lx}{\mathcal{L}\mathfrak{x}}
\newcommand{\Lu}{\mathcal{L}\mathfrak{u}}

\newcommand{\sgeq}{\succcurlyeq}

\newcommand{\CR}[1]{{#1}}

%% file: Inputs/Abstract.tex
\begin{abstract}                
	This paper addresses finite-time horizon optimal control of single-loop networked control systems with stochastically modeled communication channel and  disturbances. To cope with the uncertainties, an optimization-based control scheme is proposed which uses a disturbance feedback and the age of information as central aspects. The disturbance feedback is an extension of the control law used for balanced stochastic optimal control previously proposed for control systems without network. Balanced optimality is understood as a compromise between minimizing of expected deviations from the reference and  minimization of the uncertainty of future states.
	Time-varying state constraints as well as time-invariant input constraints are considered, and the controllers are synthesized by semi-definite programs.
\end{abstract}

%% file: Inputs/Introduction.tex
The increasing importance of networked control systems (NCS) is due to two reasons: First, wireless sensors have become widely available and simpler to integrate in control systems, enabling the control of non-stationary processes in single loop, see e.g. \cite{4118465}. Secondly, the need for controlling  large systems with multiple agents has promoted the development of distributed control strategies. On the one hand, these reduce the computational complexity of solving the overall control problem (thus alleviating to meet real-time requirements), while principles of optimality can be maintained \cite{6415463}. On the other hand, they also require information flow between the agents. From the perspective of a single agent, the reception of information from a neighbor can be seen as obtaining sensor information through wireless communication. In consequence, many methods provided for single-loop loops consisting of plant, sensor, communication unit, and controller can be transferred to larger distributed networks -- thus, the motivation of this paper is to propose
an optimal control scheme for a single-loop structure, however, with the perspective to be used for more general system structures and for receding horizon control in upcoming work.

Any wireless communication network with possible imperfections (such as packet loss and/or latencies) can be interpreted and modeled as stochastic process. Consequently, the control loop is subject to this uncertainty, and the imperfections of communication may endanger the stability of the controlled system, at least if the resulting delays are in the order of the dominant time constants of the plant. Then, the consideration of the dynamics of the communication network within controller synthesis is mandatory, and the joint design of the network operation and the controller appears as a promising concept \cite{9146991,hahn2018distributed}. 

In contrast to the latter two approaches, this paper addresses the finite-time horizon stochastic optimal control of  uncertain linear systems. Derived from the methods in \cite{1185110,4738806} and \cite{asselborn2015probabilistic}, the recently presented scheme of \emph{balanced stochastic optimal control} in \cite{WC20} determines a compromise between rejecting disturbances and meeting goals of optimal control. Thus, uncertainties arising in predicting the system behavior can be effectively minimized, while the convergence of the expected states to the origin is still guaranteed. The disturbance rejection is formulated with respect to an optimization over feedback policies, hence the satisfaction of state and input constraints in a probabilistic sense (also called \emph{chance constraints}) can be considered as well.
However, the control scheme in \cite{WC20} is not applicable to networked control systems with uncertain availability of information. To handle this, the work in \cite{wu2010stability,zhivoglyadov2003networked} uses a policy which switches between open-loop and closed-loop control depending on to the availability of information. There, the stability is guaranteed for certain assumptions on the communication network, but the satisfaction of state and/or input constraints formulated for the plant were not addressed. (This holds true for the vast majority of papers on feedback control of networked control systems.)

In contrast, the present paper combines the two mentioned aspects: the balanced stochastic optimal control of \cite{WC20}, which optimizes over a control policy to consider constraints as well as the reduction over conservatism, and the use of a switching control policy for the communication network similar to \cite{zhivoglyadov2003networked}.
The paper is organized as follows: Sec. 2 clarifies the notation used throughout the paper, and Sec. 3 specifies the control problem. Sec. 4 explains how the \emph{age of information} can be used efficiently to encode the effect of the communication network on the plant control, and the Sec. 5 as main part shows how this measure can be used within controller synthesis. Section 6 provides a numerical example, and Sec. 7 concludes the paper.

%% file: Inputs/Preliminaries.tex
This section introduces a part of the notation used in the sequel, and it refers to some mathematical facts required in the upcoming sections.

Let $s(k)$ denote the discrete-time value of a vector $s(t)\in\mathbb{R}^n$ for time $t=t_0+k\cdot\Delta t$ with $k\in\mathbb{N}^+$ and a constant time step $\Delta t\in\mathbb{R}^+$. If $s(k)$ is predicted in a previous point of time $l$, the notation $s_{k|l}$ is used. However, for brevity of notation, let $t_0=0$ be frequently used as the time instant of prediction, and then $s_k:=s_{k|0}$ is used for short for the predicted value.

The symbol $\predH$ denotes the set $\setdef{k\in\mathbb{N}_{\geq0}}{k \leq H-1}$ with $H$ denoting the prediction horizon.
Bold letters denote matrices collecting sub-matrices of a  signal predicted along the horizon, e.g. \mbox{$\textbf{s} =\begin{sbm}s_0^T,s_1^T,\dots,s_H^T\end{sbm}^T$}, and the operator $\row_k(\cdot)$ is used to select the $k$-th row of such an matrix, e.g. \mbox{$\row_{3}(\textbf{s})=s_{3} $}.

For a vector $s\in\mathbb{R}^n$, a convex polytope: 
\begin{align}
	\mathbb{S}=\setdef{s}{C_s \cdot s \leq b_s}, 
\end{align}
is denoted by $(C_s,b_s)$, and $n_\mathbb{S}$ is the number of its faces.

An ellipsoidal set is defined by a center point $\bar{s}$ and a shape matrix $\Sc $. 
The affine transformation of an ellipsoidal set with matrix $M$ and vector $v$ is again an ellipsoidal set according to:
\begin{align}
M\cdot\ellipsoid{\bar{s}}{\Sc }+v = 	\ellipsoid{M\bar{s}+v}{M\Sc M^T}. \label{eq:affinetrans}
\end{align}

A multivariate normal distribution of an $n$-dimensional random vector $s$ with covariance matrix $\Sc \in\mathbb{R}^{n \times n}$, and mean value $\bar{s}\in\mathbb{R}^n$ is written as $s\sim\NV{\bar{s}}{\Sc}$, and
a Bernoulli distribution of a random vector $s$ with probability matrix $p$ is referred to by $s\sim\B{p}$.
The expected value of a random value $s$ is denoted by $\E s = \bar{s}$, and the covariance by $\cov[s] = \Sc$. The expected value is always based on the information available in $t_0$.

The sum of two normally distributed random variables $s_1\sim\NV{\bar{s}_1}{\Sc _1}$ and $s_2\sim\NV{\bar{s}_2}{\Sc _2}$ is again normally distributed:
\begin{align}
	s_1+s_2\sim\NV{\bar{s}_1+\bar{s}_2}{\Sc _1+\Sc _2}. \label{eq:stillnd}
\end{align}
The level curves of a Gaussian probability density function are ellipsoidal. Throughout the paper, the mean value of a normal distribution coincides with the center point of the confidence ellipsoid, and the shape matrix is equal to the covariance matrix of the distribution, thus the same notation for these quantities is used.

The symbol $\norm{s}_Q=s^T\cdot Q \cdot s$ denotes a weighted 2-norm of a vector $s$ with symmetric and positive semi-definite weight matrix $Q$.



%% file: Inputs/ProbDef.tex

The sturcture of the system class under consideration is shown in \figref{fig:sys} and consists of a discrete-time linear system with probabilistically modeled additive disturbances, a controller, and a communication network. The systems dynamics is modeled as:
	\begin{align}
		&x(k+1) = Ax(k)+Bu(k)+Ew(k) \label{eq:state_eq},\\
		&w(k)\sim\NV{0}{\Wc(k)}, \label{eq:distdef}
	\end{align}
	where $x\in \mathbb{R}^{n_x}$ is the state vector, $u\in \mathbb{R}^{n_u}$ the input vector, and $w\in \mathbb{R}^{n_w}$ the disturbance vector. The additive disturbances are assumed to be i.i.d. and normally distributed with zero-mean and covariance matrix $\Wc$.
	\begin{figure}[t!]
		\begin{subfigure}[b]{0.75\linewidth}
			\centering
			\psfrag{w}[r][r]{$w(k)$}
			\psfrag{u}[r][r]{$u(k)$}
			\psfrag{a}[c][c]{Plant}
			\psfrag{c}[c][c]{Controller}
			\psfrag{x}[r][r]{$x(k)$}
			\psfrag{S}[c][c]{S}
			\psfrag{R}[c][c]{R}
		    \psfrag{com}[c][c]{com.}
		    \psfrag{net}[c][c]{net.}
			\psfrag{min}[c][c]{}
			\includegraphics[width=0.99\linewidth]{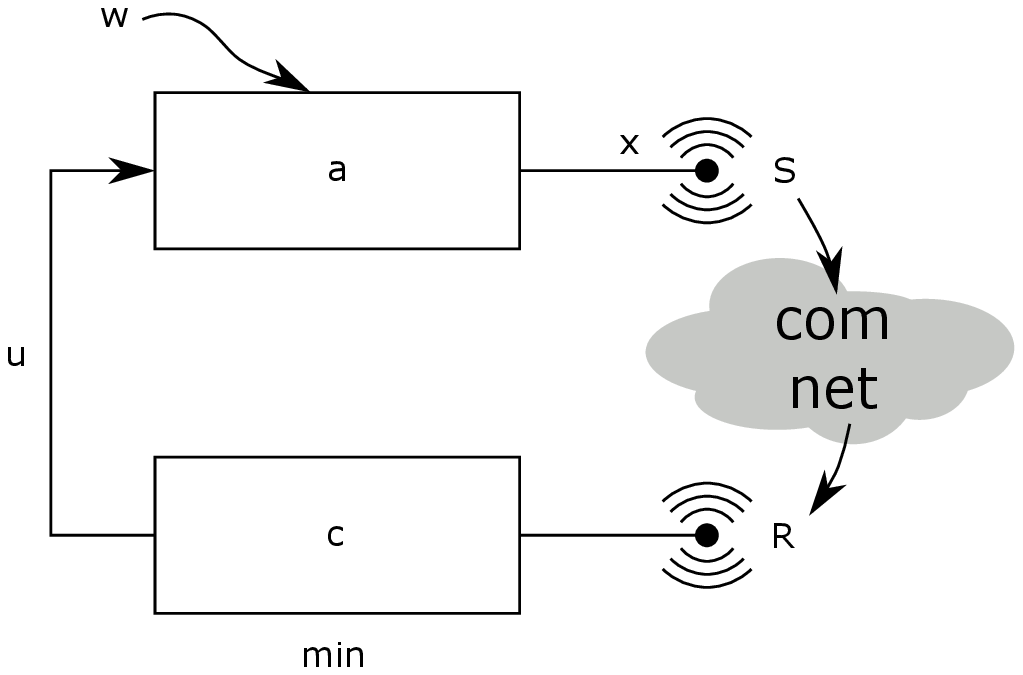}
			\subcaption{Overall system.}
			\label{fig:sys}
		\end{subfigure}%
		\hfil
		\begin{subfigure}[b]{0.25\linewidth}
			\centering
			\psfrag{A}[c][c]{\textfrak{S}}
			\psfrag{B}[c][c]{\textfrak{R}}
			\psfrag{x0}[c][c]{$x(k)$}
			\psfrag{p0}[l][l]{$\p(k)$}
			\psfrag{0p}[c][c]{$1-\p(k)$}
			\psfrag{k0}[c][c]{}
			\includegraphics[width=0.9\linewidth]{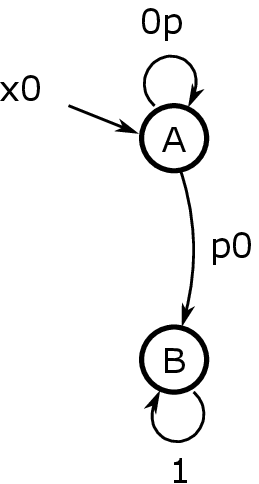}
			\smallskip
			\subcaption{Communication network}
			\label{fig:comnet}
		\end{subfigure}%
		\caption{Structure of the considered type of networked control system and the communication system.}
		\label{fig:setup}
		\hfil
	\end{figure}{}
	
	\begin{assum}\label{assum:NV}
	    The initial state $x_0:=x(t_0)$ is known. The stochastic process of the disturbance may vary over time, but is assumed to be known as $\Wc_k:=\Wc(k)$. 
	\end{assum}
	
	The state $x(k)$ and input $u(k)$ have to satisfy polytopic chance-constraints with a probabilities given by $\dx$, and $\du$ respectively (both typically chosen to be close to 1):
	\begin{align}
		&\prr{x(k)\in\mathbb{X}_k}\geq \dx, \;\; \mathbb{X}_k=\setdef{x}{C_{x_k} \cdot x \leq b_{x_k}}, \quad \label{eq:stateconstraint}\\
		&\prr{u(k)\in\mathbb{U}}\geq \du, \quad\;\; \mathbb{U}=\setdef{u}{C_u \cdot u \leq b_u}. \label{eq:inputconstraint}
	\end{align}
	The sets $\mathbb{X}_k$ and $\mathbb{U}$ are assumed to be convex and compact, and contain the origin in their interior for all $k$.

	The controller is connected to the plant in wired form, but the link between the sensor and the controller is established by wireless communication. Assume here that the communication network  is a simple one-link network and consists of a sender S, a receiver R, and the link-probability $\p(k)$. The link probability can be understood as an abstract representation of the network characteristics (or protocol), as e.g. resending information in case of failed transmission, or activation/deactivation of the channel by a network controller, etc. The simplicity of the communication network is here chosen for brevity of explanation, but the control synthesis described in the following is not limited to this network structure.
	Assuming a time-varying link-probability, the communication network is modeled by the Markov chain  for each time-step $k$, as shown in \figref{fig:comnet}.
	
	\begin{assum}\label{assum:link_prob}
		The link-probability $\p(k)$ may vary over time, but is assumed to be known (possibly according to a predictive network control scheme, see  \cite{8442615}), implying that $p_k:=p(k)$ holds true. \CR{Further, possible information dropout at time $k$ and $k+t$ are i.i.d. for all $k$ and $t$.}
	\end{assum}

    In time $k$, the current state $x(k)$ is available to S. The information is broadcasted to R (and thus available to the controller) with probability $\p(k)$. If information is not received, the information remains with  S, and the Markov chain remains in the corresponding state \textfrak{S}.
	Independently of whether information is received, the controller has to apply a control input $u(k)$ to the system, which is then simultaneously affected by the disturbance $w(k)$.\\
	In $k+1$, a new state $x(k+1)$ is measured, and if the previous transfer failed, old and new state information is available to S.
	\begin{assum}\label{assum:accum_sendig}
		The communication network uses an acknowledgment"=based protocol like TCP, thus the S knows whether information is received by R, or not. In case the information failed to be sent in time $k$, then old and new information are transmitted together in $k+1$ with link probability $\p(k+1)$.
	\end{assum}
	\begin{rem}
	    Assumption \ref{assum:accum_sendig}, which will be necessary to reconstruct disturbances affecting the dynamics previously,  may imply an arbitrarily large package size in recursive execution -- this, however, would only pose problems for the unrealistic case of persistently low values of $\p(k)$.
	\end{rem}
	
	Similarly to the control scheme presented in \cite{WC20}, this paper uses time-varying feedback control laws, aiming at the determination of an admissible control sequence for the horizon $H$:
	\begin{align}
	    \uf = \begin{bmatrix} u_0^T & u_1^T & \dots & u_{H-1}^T
	    \end{bmatrix}^T.
	\end{align}
	The following description distinguishes between an offline controller design (where disturbances, communication links, states etc. are modeled stochastically), and the online use of the control law (where specific values for disturbances, communications, etc. are measured and processed).
	According to Assumption \ref{assum:NV}, the initial state $x_0$ is known, but the disturbance is normally distributed, i.e., deterministic values for future states are not available when the control law is determined. Nevertheless, one can predict the behavior of the state $x\kps1$ under the impact of the control input $u\kps{0}$ and the disturbance $w\kps{0}\sim\NV{0}{\Wc_0}$ according to \eqref{eq:state_eq}. The latter is again normally distributed with respect to \eqref{eq:affinetrans} and \eqref{eq:stillnd}:
	\begin{align}
		x\kps{1}=Ax_0+Bu\kps{0}+Ew_0 \sim\NV{\bar{x}\kps{1}}{\Xc \kps{1}} \label{eq:statex1NV}.
	\end{align}
	Due to recursive computation, the following states $x_k$, $k>1$  are normally distributed, too.
	The main control objective is to steer the mean value of the state to the origin with acceptable costs for the input, while the state uncertainty is minimized at the same time. The	minimization of the uncertainty is equivalent to the minimization of the volume of the confidence ellipsoids for the states (see Sec. 2), and thus to minimize the set of states which are reachable with a certain probability.
	
	As the availability of information on the system state to the controller is a central point in this paper, this aspect is described in more detail in the upcoming section.

%% file: Inputs/AgeOfInformation.tex

Similar to the previous work in \cite{hahn2019robust}, the following definition of AoI is employed:
\begin{defn}
    The quantity $a(k)\in\mathbb{N}_{\geq0}$ denotes the age of the newest state information $x(l),\; l\in\mathbb{N}_{\leq k}$ available to the controller, and it quantifies the difference between the time-instances of sending and using $x(l)$.
\end{defn}
\begin{exmp}
    According to Assumption \ref{assum:NV}, the initial state $x_0$ is known to the controller. If then the communication link fails for the next five time-steps, i.e. for $k\in\{1,\ldots,5\}$, the newest information is still $x_0$, and $a(5)=5$. If the next transmission is successful, the AoI in $k=6$ equals zero, i.e. $a(6)=0$.
\end{exmp}
\begin{rem} \label{rem:AoI}
	The Age of Information is defined with respect to the information of the state $x(k-a(k))$. In consequence of Assumption \ref{assum:accum_sendig} and \eqref{eq:state_eq}, the controller can reconstruct each disturbance up to $w(k-a(k)-1)$ in time-step $k$. Hence with $a(k)=0$, the newest disturbance possibly available to the controller in $k$ is $w(k-1)$.
\end{rem}
Similar to the prediction of future states, there are no deterministic values for the AoI available at the time-instance of planning, but the behavior of the AoI is predictable according to the Markov chain shown in \figref{fig:AoI-MC} with states $\sigma_j \rightarrow a_k=j$, and the probability \mbox{$q_k=1-\p_k$} for incrementing the AoI.
\begin{figure}[h!]
    \centering
    \psfrag{a}[c][c]{$\sigma_0$}
    \psfrag{b}[c][c]{$\sigma_1$}
    \psfrag{c}[c][c]{$\sigma_2$}
    \psfrag{d}[c][c]{$\sigma_3$}
    \psfrag{dots}{\dots}
    \psfrag{p}{$\p_k$}
    \psfrag{q}{$q_k$}
    \includegraphics[width=\linewidth]{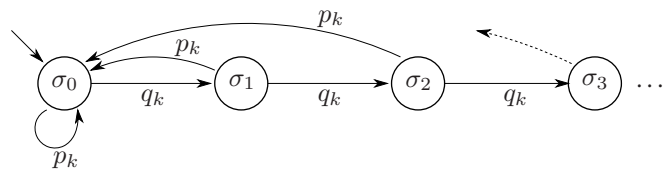}
    \caption{Exemplary Markov chain to model the AoI for 3 time-steps.}
    \label{fig:AoI-MC}
\end{figure}
This Markov chain corresponds to the communication network shown in \figref{fig:comnet}, and is exemplary for 3 time-steps, since the maximum AoI is three according to $\sigma_3$. For more time-steps, the model is extended to the right and ends with a self-loop on the last state. With time-variant communication link probabilities, the transition matrix $P_k$ for this model is time-variant, too:
\begin{align}
	P_k = \begin{bmatrix}
	p_k & p_k & p_k & 1\\
	q_k & 0 & 0 & 0\\
	0 & q_k & 0 & 0\\
	0 & 0 & q_k & 0
	\end{bmatrix}.
\end{align}
Initialized to $\mu_0=[1\;\;0\;\;0\;\;0]^T$, the probability vector $\mu_k$ for the AoI evolves according to:
\begin{align}
    \mu_k=\left(\prod_{l=1}^{k} P_l\right)\cdot \mu_0 \label{eq:mu},
\end{align}
and the probability for the age of information is given by:
\begin{align}
    \prr{a(k)=j}=\mu_k[j] \label{eq:prop_mu}
\end{align}
with $\mu_k[j]$ being the $(j+1)$-th entry of the vector $\mu_k$.


%% file: Inputs/ControllerSynthesis.tex
This section describes the synthesis of suitable feedback policies to control the system in presence of stochastic AoI, and it explains how the state uncertainty can be formalized. 

\subsection{Feedback policy}    
On one hand, if the state is available to the controller, the optimal control strategy for deterministic and constrained linear systems is based on time-varying state feedback control laws, see e.g. \cite{borrelli2003constrained}. On the other hand, if state information is not available, the optimal control law should be based on the feedback of an estimated value (e.g. a previously predicted value), i.e. a feedback of expected values is realized.
    In general with uncertain availability of the state information to the controller, a switched control policy of time-varying feedback laws can thus be proposed to combine the two cases:
    \begin{salign}
        u(k) = \sum_{r=0}^{k}K_{k,r} \cdot \tilde{x}_r,\quad \tilde{x}_r=\begin{cases}
        x(r) & \text{if $x(r)$ is available \;} \\
        \bar{x}_{r|l} & \text{otherwise}
        \end{cases} \label{eq:statefeedback}
    \end{salign}
	Here, $l:=k-a(k)$ denotes for a moment the time-step of the last certain information, while $\bar{x}_{r|l}$ is the state predicted in $l$ for a time $r> l$ of the nominal system:
	\begin{align}
	    \bar{x}_{r|l}=A^{r-l} x(l)+\sum_{r=1}^{a(k)} A^{a(k)-r}Bu{(l+r-1)}. \label{eq:expectedstate}
	\end{align}
	In other words, $\bar{x}_{r|l}$ is the expected state for time $r$ predicted or reconstructed from the last certainly known state $x(l)$.
	
    \begin{rem}
        To keep the controller synthesis as simple as possible, the paper aims at using the same feedback matrix $K_{k,r}$ in \eqref{eq:statefeedback} regardless the availability of the state $x(r)$.
        Hence, the offline synthesis aims at determining the feedback matrices. When using these online, the controller decides whether the measured and communicated state (if available) or a predicted state has to be fed back according to \eqref{eq:statefeedback} with feedback matrix $K_{k,r}$.
    \end{rem}

    In general for systems without communication, a state feedback can be reformulated into a feedback of initial state and disturbances (the interested reader is referred to \cite{goulart2006optimization} and \cite[Lemma~2]{WC20}). Motivated by and similar to these reformulations of control laws, the following is stated:
    \begin{lem} \label{lem:controllaw}
        The state feedback \eqref{eq:statefeedback} with \eqref{eq:expectedstate} can be reformulated into a disturbance feedback:
        \begin{align}
            u(k)=V_k\cdot x_0+ \sum_{r=0}^{k-1} \ident_{k,r}\cdot M_{k,r} \cdot w(r) \label{eq:inddf}
        \end{align}
        with feedback matrices $V_k\in\mathbb{R}^{n_u \times n_x}$ and $M_{k,r}\in\mathbb{R}^{n_u\times n_w}$ and with an indicator function depending on the AoI: 
        \begin{align}
            \ident_{k,r}=\begin{cases} 1 & \forall\; r < k-a(k)\\ 0 & \forall\; r\geq k-a(k) \end{cases}. \label{eq:ident}
        \end{align}
    \end{lem}
    \begin{proof} see Appendix \ref{app:proofcontrollaw}.
    \end{proof}

    Control law \eqref{eq:inddf} with \eqref{eq:ident} implies the feedback of all disturbances $w(k-a(k)-1)$, which are available to the controller in time-step $k$.
    According to Assumption \ref{assum:accum_sendig} the controller can reconstruct all disturbances $w(k-a(k)-1)$, see Remark \ref{rem:AoI}.

    \begin{rem} \label{rem:indicator}
    	A common disturbance feedback as used in (\cite{gross2014distributed,goulart2006optimization,WC20}, etc.), and the proposed disturbance feedback \eqref{eq:inddf} differ only wrt. the indicator function $\ident_{k,r}$, and	
    	the indicator function only depends on the AoI $a(k)$. Thus, (ex)changing the communication network in the setting \figref{fig:setup} results only in an adaption of the Markov chain in \figref{fig:AoI-MC} used to model the AoI. Consequently, the prediction of the systems behavior and the controller synthesis described in the following are decoupled from the exact model of the communication network, but requires a prediction of the AoI.
    \end{rem}
    
    \subsection{Prediction of the closed-loop behavior}
    To plan the systems behavior over a horizon of length $H$, the notation of stacked vectors for the states, inputs, and disturbances is used:
    \begin{align}
    	\xf = \begin{sbm}
    	x_0\\
    	x\kps1\\
    	\vdots\\
    	x\kps{H}
    	\end{sbm},\quad \uf =\begin{sbm}
    	u\kps0\\
    	u\kps1\\
    	\vdots\\
    	u\kps{H-1}
    	\end{sbm},\quad \wf=\begin{sbm}
    	w\kps0\\
    	w\kps1\\
    	\vdots\\
    	w\kps{H-1}
    	\end{sbm}.
    \end{align}
    With matrices $\Af$, $\Bf$, and $\Ef$ of adequate size (see Appendix \ref{app:vectornotation}), the evolution of the states over the horizon is encoded by:
    \begin{align}
    	\xf &=\Af x_0 + \Bf \uf + \Ef \wf . \label{eq:statepred}
    \end{align}
	Similarly as in \cite{WC20}, a stacked form of the disturbance feedback \eqref{eq:inddf} is introduced
	(with feedback matrices $(\Vf,\Mf)$ as in Appendix \ref{app:vectornotation}) to predict the closed-loop behavior in the following. The indicator function \eqref{eq:ident} is modeled as random variable with Bernoulli distribution to model the AoI:
   	\begin{align}
   		\ident_{k,r}\sim\B{p_{k,r}}, \quad p_{k,r}=\sum_{l=0}^{k-r-1}\mu_k[l], \label{eq:bernoulli}
   	\end{align}
   	see Appendix \ref{proof:bernoulli}. Collecting all indicator functions over the horizon into an indicator matrix $\Pf\in\mathbb{R}^{H\cdot n_u \times H\cdot n_w}$ (which has the same size as the feedback matrix $\Mf$):
	\begin{align}
		\Pf=\begin{sbm}
			0 & 0 &\dots & 0\\
			\ident\kPs{1}{0} & 0 &\dots & 0\\
			\vdots & \ddots & \ddots & \vdots\\
			\ident\kPs{H-1}{0} & \dots & \ident\kPs{H-1}{H-2} & 0\\
		\end{sbm}\otimes \textbf{1}_{n_u\times n_w}, \label{eq:structureP}
	\end{align}
	the control trajectory results in:
	\begin{align}
		\uf=\Vf\cdot x_0 + \left(\Pf\circ \Mf\right)\cdot \wf \label{eq:inputpred}.
	\end{align}
	With the Hadamard product $\circ$, the indicator matrix $\Pf$ selects the entries of $\Mf$ with which all disturbances available to the controller can be fed back. With \eqref{eq:statepred} and \eqref{eq:inputpred}, the closed-loop dynamics of the system follows to:
	\begin{align}
		\xf&=(\Af+\Bf\Vf)x_0+\left(\Ef+\Bf\left(\Pf\circ\Mf\right)\right)\wf. \label{eq:statepred_vec}
	\end{align}

	The uncertainty caused by the communication network (introduced by $\Pf$) only affects the rejection of the disturbance $\wf$. With stochastically independent $\Pf$ and $\wf$, the state prediction $\xf$ is a mixed distribution, and a vector of expected values $\E{\xf}$ is defined. For the Bernoulli distribution, all possible cases are considered, one of which is:
	\begin{align}
	    \Pf\s{i}\in\Pf\s{\Omega}=\{\Pf\s{1},\dots,\Pf\s{o}\} \label{eq:PfOmega}
	\end{align}
	with $\Pf\s{\Omega}$ denoting the set of all possible realizations and $o$ its cardinality. The state trajectory can be written as function depending on the Bernoulli case:
	\begin{align}
	\xf =  \begin{cases}
	\xf\s{1} & \text{if } i=1\\
	\xf\s{2} & \text{if } i=2\\
	\vdots\\
	\xf\s{o} & \text{if } i=o
	\end{cases}
	\end{align}
	with:
	\begin{align}
		\xf\s{i}&=(\Af+\Bf\Vf)x_0+\left(\Ef+\Bf\left(\Pf\s{i}\circ\Mf\right)\right)\wf, \label{eq:statepred_i}\\
				&=: \Am x_0 + \Em\s{i} \wf,
	\end{align}
	and closed-loop matrices $\Am$ and $\Em\s{i}$.
	All disturbances $w_k$ are assumed to be iid., i.e. $\wf$ is a vector of normally distributed disturbances. Thus, each state $\xf\s{i}$ is normally distributed for $k>0$, too,  and can be referred to by expected values and covariance matrices representing the state uncertainty.

	\begin{rem}\label{rem:complexity}
		The minimization of the uncertainty of $\xf$ could be realized by minimizing the uncertainty of each state $\xf\s{i}$ separately. But for increasing $H$, this would lead to large complexity, since each realization $\Pf\s{i}$ results in a different input trajectory $\uf\s{i}$ according to \eqref{eq:inputpred}, and the maximum number of possible realizations $o$  is given by the Catalan number $o\leq C_H=\frac{(2H)!}{H!(H+1)!}$. Thus, e.g., the one-link communication network given in \figref{fig:comnet} leads to \mbox{$o=2^{H-1}$}.
		In addition, the communication network will evaluate over time, thus even if all possible cases are minimized (and a set of input trajectories $\uf\s{i}$ is obtained), the question of which $\uf\s{i}$ the optimal one is could not be answered till time reaches the end of the prediction horizon.
	\end{rem}
	
	Since neither the consideration of each possible case obtained by the Bernoulli distribution nor the consideration of an \textit{averaged} case is advisable, the complexity is here reduced by considering one specific case of the Bernoulli distribution. 
	This case is denoted by the superscript $\beta$ (leading to $\xf\s{\beta}$, $\uf\s{\beta}$, and $\Pf\s{\beta}$), an it is described in detail in the following section.
	
\begin{figure*}[b]
	\rule{\linewidth}{0.1pt}\smallskip
	\begin{salign}
		\Lu_k\hs{i}=\begin{bm}
			(b_u\hs{i}-C_u\hs{i}V_k x_0) & C_u\hs{i} c_u^\frac{1}{2} (\Pf\s{\beta}_{k} \circ \Mf_{k})\Wf\\
			\star & (b_u\hs{i}-C_u\hs{i}V_k x_0)\Wf
		\end{bm}\sgeq0,\quad
		\Lx_{k+1}\hs{j} = \begin{bm}(b_{x_{k+1}}\hs{j}- C_{x_{k+1}}\hs{j}\Am_{k}x_0) & C_{x_{k+1}}\hs{j} c_x^\frac{1}{2} \Em_{k}\s{\beta}\cdot\Wf\\
			\star & (b_{x_{k+1}}\hs{j}-C_{x_{k+1}}\hs{j}\Am_{k}\cdot x_0)\Wf \end{bm}\sgeq 0	\quad \label{eq:LMIconst}
	\end{salign}
\end{figure*}	
	
\subsection{State and Input constraints}
   This subsection first clarifies the consequences considering one predefined case $\Pf\s{\beta}$, and afterwards, how this case is chosen.\\
	The constraints \eqref{eq:stateconstraint} and \eqref{eq:inputconstraint} require that each state $x(k)$ and input $u(k)$ are contained in their admissible sets $\mathbb{X}$ and $\mathbb{U}$ at least with likelihood $\dx$ and $\du$.
	With choosing to optimize over one predefined control sequence according to $\Pf\s{\beta}$, 
	there are two possibilities in each time-step $k$:
	i) All disturbances, which were expected to be available by $\row_k{\Pf\s{\beta}}$ in $k$, are available to the controller, and the corresponding control law $u_k\s{\beta}$ is applicable. Even if more disturbances are available than used with $u_k\s{\beta}$, the control law \textit{can} be chosen to $u(k):=u_k\s{\beta}$. If it \textit{is} chosen to $u(k):=u_k\s{\beta}$, then, the state evolves according to $x_k\s{\beta}$, and the complexity caused by the number of different state behaviors collapses.
	ii) Otherwise, some required disturbances are missing and $u_k\s{\beta}$ is not applicable. Then only a tailored control law (so $u(k)$ as in \eqref{eq:inddf}, and not $u_k\s{\beta}$) is applicable, and the satisfaction of state and input constraints cannot be guaranteed. 
	
	Now, the main idea is the following: In each time-step $k$, \CR{adapt} the likelihoods to satisfy the state and input constraints $\dx$ and $\du$ according to:			
	\begin{align}
		\du = \gamma_k\s{u}\cdot \alpha_k,\quad  \dx = \gamma_k\s{x}\cdot \CR{\prod_{t=1}^{k-1}\alpha_t}, \label{eq:tailoredlikelihoods}
	\end{align}
	where $\alpha_k$ is the probability that the determined control law $u_k\s{\beta}$ is applicable (case i)), and where $\gamma_k\s{u}$ and $\gamma_k\s{x}$ are \textit{tailored likelihoods}.
	\CR{Note, that the probability for occurrence of state $x_k\s{\beta}$ is stochastically depended of all prior $\alpha_t,\;t<k$, thus the product is used.}
	Then, state and input constraints under the condition of control law $u_k\s{\beta}$ are satisfied at least with probability $\gamma_k\s{x}$ and $\gamma_k\s{u}$, and the satisfaction of  \eqref{eq:stateconstraint} and \eqref{eq:inputconstraint} are guaranteed.\\
	Note that the tailored likelihoods in \eqref{eq:tailoredlikelihoods} depend on $\alpha_k$ and case $\Pf\s{\beta}$, which are defined in the following statements:
    Let $\Pf\s{\beta}$ collect all indicator functions according to:
	\begin{align}
		\ident_{k,r}\s{\beta}=\begin{cases}
		1 & \text{if} \quad p_{k,r}\geq \alpha_k\\
		0 & \text{else.}
		\end{cases}, \label{eq:ident_beta}
	\end{align}
	and $\Wf= \blkdiag(\Wc_0,\dots,\Wc_{H-1})$ all covariance matrices over the prediction horizon, then the following is stated:	
	\begin{prop}\label{lem:constraints}
		Let $(C_u\hs{i},b_u\hs{i})$ and $(C_{x_{k+1}}\hs{j},b_{x_{k+1}}\hs{j})$ denote the $i$-th half-space of the polytopic input set $\mathbb{U}$, and the $j$-th half-space of the polytopic state set $\mathbb{X}_{k+1}$. Then, the state constraints \eqref{eq:stateconstraint} and the input constraints \eqref{eq:inputconstraint} are satisfied, if the LMIs in \eqref{eq:LMIconst} hold when using:
		\begin{align}
			&\Mf_k=\row_k(\Mf),\quad  \Pf_k\s{\beta}=\row_k(\Pf\s{\beta}),\\
			& \Am_k = \row_k(\Am),\quad \Em_k\s{\beta} = \row_k(\Em\s{\beta}),\\
			&\{c_u,c_x\} = (F_\chi^2)^{-1}(\{\gamma_k\s{u},\gamma_k\s{x}\},\{n_u,n_x\}),
		\end{align}
		for each $i\in\{1,\dots,n_\mathbb{U}\}$ and $j\in\{1,\dots,n_{\mathbb{X}_{k+1}}\}$, and if $\alpha_k$ is chosen to:
		\begin{align}
			\alpha_k &= \min_{r\in R_k} \{p_{k,r},1\},\; R_k=\setdef{r\in\mathbb{N}_{\geq0}}{p_{k,r}> \CR{\du}}.\\
			&\CR{\st \prod_{t=1}^{k-1}\alpha_t\geq \dx}\\[-5ex]
		\end{align}
	\end{prop}
	\begin{proof}
		See Appendix \ref{proof:constraints}.
	\end{proof}
	
	Note that the above choice of determining $\alpha_k$ is not the only possible one, and other options
	satisfying $\alpha_k>\max\{\dx,\du\}$ are conceivable.

\subsection{Determination of feedback matrices by SDP}
	For minimizing the state covariance matrices, an auxiliary matrix $\Sf=\Sf^T\sgeq0$ is introduced (motivated by the scheme in \cite{asselborn2015probabilistic}), which determines an upper bound for the cross-covariance of the state trajectory $\xf\s{\beta}$:
	\begin{align}
	\Sf\sgeq \cov\left[\xf\s{\beta}\right] = \Em\s{\beta} \Wf \Em\s[T]{\beta}
	\end{align}
	formulated by the following LMI:
	\begin{align}
	\Ls = \begin{bmatrix}
	\quad \Sf \quad \null & \Em\s{\beta}  \Wf\\
	\star & \Wf
	\end{bmatrix}\sgeq 0.
	\end{align}
An SDP problem with stacked weights $\mathcal{Q}$, and $\mathcal{R}$ is stated as: 
\begin{align}
	&\min_{\Vf,\Mf} \norm{\E\xf}_\mathcal{Q} + \norm{\E\uf}_\mathcal{R} + \norm{\trace\left(\Sf\right)}_\mathcal{S} \\
	\text{s.t.:} \quad & \E\xf=(\Af+\Bf\Vf)x_0,\quad \E\uf=\Vf x_0, \label{prob:2}\\
	&\Ls\sgeq 0,\\
	&\Lx\hs{i}\kp1\sgeq  0 \quad  \forall\; i\in\{1,\dots,n_{C_{x_{k+1}}}\}, \forall\; k\in\predH,\\
	&\Lu\hs{i}\kps{k}\sgeq 0 \quad \forall\; i\in\{1,\dots,n_{C_u}\}, \forall\; k\in\predH,\\
	&\mathcal{Q}=\blkdiag(Q_0,\dots,Q_H),\\
	&\mathcal{R}=\blkdiag(R_0,\dots,R_{H-1}),
\end{align}
with symmetric and positive semi definite weights $Q_k$, $R_k$, $\mathcal{S}$. The optimal solution of this problem is denoted by $(\Vf^\star,\Mf^\star)$, and an optimal state feedback matrix $\Kf^\star$ follows to specify the control law \eqref{eq:statefeedback} (see \cite[Theorem~7]{WC20}). In case, more information is available in $k$ than expected according to $\Pf\s{\beta}$, the corresponding feedback matrices $M_{k,r}$ are chosen to zero in the optimization, such that \mbox{$(\Pf\circ\Mf)=(\Pf\s{\beta}\circ\Mf)$} holds, thus guaranteeing the satisfaction of the constraints.

%% file: Inputs/Simulation.tex
To illustrate the proposed method, it is applied to a severely disturbed example with the dynamics of type \eqref{eq:state_eq} using:
\begin{salign}
	A{\small =}\begin{sbm}
		1 &  0.3 \\ 0 & 1
	\end{sbm},\, \quad 
	B{\small =}\begin{sbm}
	0.045\\ 0.3
	\end{sbm},\,\quad 
	E{\small =}I_2,
\end{salign}
and the initial state and disturbance distributions:
\begin{align}
	x(t_0)=\begin{sbm}
		85\\-5
	\end{sbm}, \;
	\Wc(k){\small=}\begin{sbm}0.1 & 0\\0&3\end{sbm}\, \forall k\in\predH.
\end{align}
The input constraint is defined by $\mathbb{U}=\setdef{u}{\norm{u}\leq30}$, and the chance constraint for this region is selected to:
\begin{align}
	\prr{u(k)\in\mathbb{U}}\geq \du =0.95.
\end{align}

A time-invariant state constraint is defined to:
\begin{align} 
	\mathbb{X}\kps{k+1}=\setdef{x}{-20\leq x_1(k) \leq 100,\\ -40\leq x_2(k) \leq 30} \;\forall k\in\predH
\end{align}
with the probability $\prr{x(k)\in\mathbb{X}\kps{k}}\geq \dx =0.9$.

\begin{figure}[t!]
	\centering
	\small
	\psfrag{x}[c][c]{state $x_1(k)$}
	\psfrag{y}[c][c]{state $x_2(k)$}
	\includegraphics[width=\linewidth]{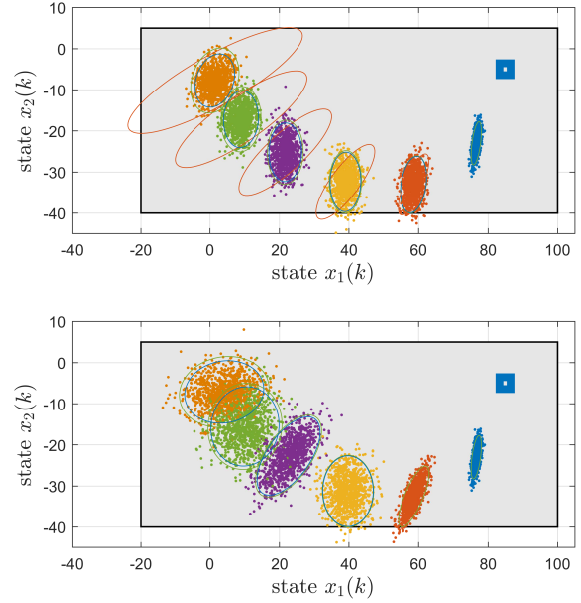}
	\caption{\label{fig:sim_state2} \small{
		Simulation results for even values of $k$, and the two parameterizations of the communication link: i) top , ii) bottom.  The initial states are shown as blue square and the admissible state space as gray box. The $\dx$-confidence ellipsoids of the state  are shown as blue ellipsoids for the case of successful transmission of any information, and as red ellipsoids for the case that communication fails in every time-step. The $\gamma_k\s{x}$-confidence ellipsoids used in the optimization problem for $\alpha_k$ are marked as green ellipsoids. In addition, the simulated states obtained for 1000 Monte Carlo simulations are shown as colored dots (different colors for selected $k$).}}
\end{figure}
To demonstrate the proposed methodology, the probability of the communication link according to \figref{fig:comnet} is chosen with two different parameterizations, a constant one and an alternating one:
\begin{itemize}
    \item[i)] $p(k) = 0.8 \;\forall k\in\predH$,
    \item[ii)] $\left[p(0) \dots p(H-1) \right] :=0.1\cdot \left[1\;\; 9\;\; 1\quad 1\;\; 9\;\; 1\quad\dots\right]$.
\end{itemize}
Note that these numbers are chosen, by intention, to model a poorly performing communication network (in fact much worse than typically observed in practice) -- this serves to demonstrate that the proposed scheme nevertheless leads to good control performance.

The control horizon is chosen to $H=12$, and the cost functional is parameterized with time-invariant weights \mbox{$Q_k=\diag(2,1)$}, $R_k=0.1$, and $\mathcal{S}=100$.
The optimization problems were solved on a \SI{3.4}{\giga\hertz} Quad-Core CPU using  \textit{Matlab2018a} and the solver \textit{Mosek}, and the computations took an average  time of around \SI{0.6}{\second}.

\begin{table}[b!]
    \centering
    \begin{tabular}{c|c|c|c|}
         $\varnothing$ MCS &$\sum \norm{x(k)}_{Q_k}$ & $\sum\norm{x(k)-\bar{x}(k)}_{Q_k}$  & $\sum \norm{\bar{x}(k)}_{Q_k}$ \\
         \hline
         version i) & 3920 & 26.8 & 5.03\\ 
         \hline
         version ii) & 3945 & 74.5 & 5.03 \\ 
         \hline
    \end{tabular}
    \caption{Performance criteria for both versions.}
    \label{tab:my_table}
\end{table}
The simulation results for the two parameterizations are shown in \figref{fig:sim_state2}.
It is obvious that the $\gamma_k\s{x}$-confidence ellipsoids are contained in the admissible state space, and that the confidence ellipsoids and the simulated states are reliably steered towards the origin. For an adequate comparison, the same values for the disturbances underlie the two sets of Monte Carlo simulations (MCS).

While for i) (top part of \figref{fig:sim_state2}) the ellipsoids are obtained smaller than in version ii) (bottom), the result of ii) is suitable, too, as the average probability for successful transmission of information is only $\bar p = 0.3\bar 6$, thus less than a half of the $\bar p$ in version i).
Note that the optimization problem considers the different performance of the communication network, such that the blue and green $\dx$-confidence ellipsoids differ in both versions.

In average over all simulated time-steps, the state $x(k)$ lies to $90.32\%$ within the $y_k\s{x}$-confidence ellipsoid for version i), and the input $u(k)$ with $96.508\%$ within the $y_k\s{u}$-confidence ellipsoid. For version ii), the probabilities are $90.48\%$ for the state, and $98.4\%$ for the input. Thus, the state and input constraint are satisfied with only very small conservatism.

For quantitative assessment, Table \ref{tab:my_table} shows three performance criteria for both network parametrizations evaluated in Monte Carlo simulation: 
The first column contains the weighted distance of all states to the origin (representing how good the control goal of approaching the origin is met), which is   similar for both versions.
The second column shows the weighted distance of all states to the corresponding expected state, thus representing the state uncertainty. Due to the poor network performance, a noticeable uncertainty of the states is obtained, where a worse performing network (according to the average link-probability $\bar p$) leads to an increase of the uncertainty of the controlled system. The last column contains the weighted distances of the expected states to the origin, which is similar for both versions.

%

%% file: Inputs/Conclusion.tex

This paper transfers a balanced stochastic optimal control scheme to a control system with LTI dynamics and embedded stochastically modeled communication link, where additive disturbances and uncertain communication has to be dealt with.
Uncertainties of the communication network are projected onto a tailored probability for the satisfaction of state and input constraints. Furthermore, balancing between the minimization of uncertainties and the minimization of expected distances to the origin is realized. 
The simulation results show, that the constraints are satisfied with nearly no conservatism, and that a poorly performing communication network only increases the uncertainty and not the performance of the control system (in the sense of state deviation from the reference).
In general, the presented optimization problem is independent of the  complexity of the communication network (such as the number of nodes or links, and possible time-variance) and its imperfections (like delay, packet loss etc.) by the use of a Markov chain modeling the Age of Information. 
In consequence, the SDP for controller synthesis results with the same complexity as in the case without networks.

This paper opens the field of balanced stochastic optimal control in the considered understanding for larger structures of networked control systems, which are subject of current and future work.

%% file: Inputs/Appendix.tex
\subsection{Vector Notation}\label{app:vectornotation}
    The matrices $\Af\in\mathbb{R}^{(H+1)n_x \times n_x}$ and $\Cf\in\mathbb{R}^{(H+1)n_x \times H\cdot n_x}$ have the following structure:
    \begin{salign}
    	\Af=\begin{sbm}
    	I_{n_x}\\A\\\vdots\\A^{H}
    	\end{sbm},\quad \Cf=\begin{sbm}
    	0 & 0 & \cdots & 0\\
    	I_{n_x} & 0 & \cdots & 0\\
    	A & I_{n_x} & \cdots & 0\\
    	\vdots& \vdots & \ddots & 0\\
    	A^{H-1} & A^{H-2} & \cdots & I_{n_x}
    	\end{sbm},
    \end{salign}
    while $\Bf\in\mathbb{R}^{(H+1)n_x \times H\cdot n_u}$ and $\Ef\in\mathbb{R}^{(H+1)n_x \times H\cdot n_w}$ are chosen to $\Bf = \Cf(I_{n_x}\otimes B)$, and $\Ef = \Cf(I_{n_x}\otimes E)$ respectively.\\
    The disturbance feedback matrices  $\Vf\in\mathbb{R}^{H\cdot n_u \times n_x}$, and $\Mf\in\mathbb{R}^{H\cdot n_u \times H\cdot n_w}$ are defined to:
    \begin{salign}
    	\Vf=\begin{sbm}
    		V\kps{0}\\
    		V\kps{1}\\
    		\vdots\\
    		V\kps{H-1}
    	\end{sbm},\quad
    	\Mf=\begin{sbm}
    		0 & 0 &\cdots & 0\\
    		M\kPs{1}{0} & 0 &\cdots & 0\\
    		\vdots & \ddots & \ddots & \vdots\\
    		M\kPs{H-1}{0} & \cdots & M\kPs{H-1}{H-2} & 0
    	\end{sbm}. \label{eq:structureM}
    \end{salign}

    \subsection{Derivation of equation \eqref{eq:bernoulli}} \label{proof:bernoulli}
    
    Each probability $p_{k,r}$ is given according to \eqref{eq:ident} with respect to the probability for the AoI \eqref{eq:prop_mu}:
    \begin{salign}
    	p_{k,r}&=\E{\ident_{k,r}}=\prr{r<k-a(k)} = \prr{a(k)\leq k-r-1} \\
    	&=\prr{a(k)=0}+ \ldots + \prr{a(k)=k-r-1}=\sum_{l=0}^{k-r-1}\mu_k[l].
    \end{salign}
    
\subsection{Proof of Lemma \ref{lem:controllaw}} \label{app:proofcontrollaw}
    First, consider the system dynamics \eqref{eq:state_eq} and the control law \eqref{eq:statefeedback}, for the case of available state information:
    \begin{salign}
        x_{k+1}=f(x_k,u_k,w_k), \quad
        u_k=\kappa_k(x_0,\ldots,x_k),
    \end{salign}
    with general functions:
    \begin{salign}
        f&: \mathbb{R}^{n_x} \times \mathbb{R}^{n_u} \times \mathbb{R}^{n_w}\rightarrow \mathbb{R}^{n_x},\\
        \kappa_k&: \mathbf{1}_{1,k+1}\otimes \mathbb{R}^{n_x}\rightarrow \mathbb{R}^{n_u} \label{eq:kappa}.
    \end{salign}
    The recursive state equation can be reformulated to a series of functions  $f_k$ and $\kappa_k$, where control laws are rewritten to  $\tilde{\kappa}_k$, e.g. for $k\in\{0,1\}$:
    \begin{salign}
        u_0&=\kappa_0(x_0),\\
        x_1&=f(x_0,u_0,w_0)=f(x_0,\kappa_0(x_0),w_0)=:f_1(x_0,w_0),\\
        u_1&=\kappa_1(x_0,x_1)=\kappa_1(x_0,f_1(x_0,w_0))=:\tilde{\kappa}_1(x_0,w_0),
    \end{salign}
	 or for arbitrary $k\in\mathbb{N}\geq 0$:
    \begin{salign}
        u_k &= \tilde{\kappa}_k(x_0,w_0,\dots,w_{k-1}) \label{eq:uk1},\\
        x_{k}&= f_k(x_0,w_0,\dots,w_{k-1}), \label{eq:xk1}      
    \end{salign}
    or for linear functions:
    \begin{salign}
        \tilde{\kappa}_k &=a\s{x}_0\cdot x_0 + a\s{w}_0\cdot w_0 + \ldots + a\s{w}_{k-1}\cdot w_{k-1},\\
        f_k &=b\s{x}_0\cdot x_0 + b\s{w}_0\cdot w_0 + \ldots + b\s{w}_{k-1}\cdot w_{k-1},
    \end{salign}
    with $a_i$ and $b_i$ \cite{WC20}.
    
    Secondly, if $a_k\geq0$, \eqref{eq:statefeedback} is used in combination with \eqref{eq:expectedstate}. Thus, with the last certain information $x_l \rightarrow l:=k-a_k$, it holds that all states $x_{l+1}$ up to $x_k$ are not available to the controller:
    \begin{salign}
        \underbrace{x_0,\dots,x_{l},}_{\text{available}}\underbrace{x_{l+1},\ldots,x_k}_{\text{not available}}.
    \end{salign}
    With \eqref{eq:statefeedback} and \eqref{eq:expectedstate}, it holds that:
    \begin{salign}
        u_k&=\kappa_k(x_0,\ldots,x_l,x_{l+1|l},\ldots,x_{k|l}), \label{eq:uk2}\\
        x_{k|l}&=f_{k,l}(x_l,u_l,\ldots,u_{k-1}) \label{eq:xk2},
    \end{salign}
    with functions $\kappa_k$ defined in \eqref{eq:kappa} and:
    \begin{salign}
    f_{k,l}:\mathbb{R}^{n_x}\times \left(\textbf{1}_{1,a_k}\otimes \mathbb{R}^{n_u}\right)\rightarrow \mathbb{R}^{n_x}.
    \end{salign}
    
    In \eqref{eq:uk2}, the states with index up to $l$ can be expressed by the functions given in \eqref{eq:xk1}. 
    For the first of the remaining states/inputs in the function arguments, it follows with \eqref{eq:xk1} that:
    \begin{salign}
        x_{l+1|l}&=f_{l+1,l}(x_l,u_l)=:\tilde{f}_{l+1,l}(x_0,w_0,\dots,w_{l-1}),\\
        u_{l+1}&=\kappa_{l+1}(x_0,x_l,x_{l+1|l})=:\tilde{\kappa}_{l+1}(x_0,w_0,\dots,w_{l-1}).
    \end{salign}
    Recursively for $k>l$, the states:
    \begin{salign}
        x_{k|l}=\tilde{f}_{k,l}(x_0,w_0,\dots,w_{l-1}),
    \end{salign}
    and the inputs:
    \begin{salign}
        u_k=\tilde{\kappa}_k(x_0,w_0,\dots,w_{l-1})  \label{eq:tidekappa}
    \end{salign}
    are obtained. Again, with linear functions $f_{k,l}$ and $\kappa_k$, the control law $\tilde{\kappa}_k$ in \eqref{eq:tidekappa} is a linear function of its arguments, thus (again with a set of parameters $a_i$) one can write:
    \begin{salign}
        u_k=a_0\s{x}\cdot x_0 + a_0\s{w}\cdot w_0+\ldots+a_{l-1}\s{w}\cdot w_{l-1}. \label{eq:proofidfb}
    \end{salign}
    Eventually, \eqref{eq:proofidfb} feeds back all disturbances $w_r$ with \mbox{$r\leq l-1=k-a_k-1 \Leftrightarrow r < k-a_k$}. With $V_k:=a_0\s{x}$ and $M_{k,r}:=a_r\s{w}$,  \eqref{eq:proofidfb} equals the disturbance feedback in Lemma \ref{lem:controllaw} for time $k$. \hfill $\qed$
    
%

\subsection{Proof of Proposition \ref{lem:constraints}}	\label{proof:constraints}

	According to Proposition \ref{lem:constraints} the co-domain of $\alpha_k$ is given by:
	\begin{salign}
		\alpha_k\in\;\left]\max(\du,\dx),\;\; 1\right],
	\end{salign}
	such that:
	\begin{salign}
		\gamma_k\s{u}=\frac{\du}{\alpha_k} \in \left[\du,\;\; 1\right],\quad 
		\gamma_k\s{x}\CR{\geq}\frac{\dx}{\alpha_k} \in \left[\dx,\;\; 1\right],
	\end{salign}
	hold. Now recall \eqref{eq:PfOmega}, and let $\Pf\s{>\beta}\subseteq\Pf\s{\Omega}$ denote a subset of indicator matrices $\Pf\s{i}$, for which all entries satisfy the inequality \mbox{$\ident_{k,r}\s{i}\geq\ident_{k,r}\s{\beta}\;\forall k,r\in\predH$}. {\color{red}
	}
	
	Then the following holds true with $\theta$ denoting the behavior of the communication network according to $\Pf\s{\theta}\in\Pf\s{\Omega}$:
	\begin{salign}
		&\prr{x_k\in\mathbb{X}_k}=\sum_{i=1}^{o}\prr{x_k\s{i}\in\mathbb{X}_k\vert \theta=i}\\
		&
		=\sum_{i=1}^o \pr{\theta=i}\cdot \prr{x_k\s{i}\in\mathbb{X}_k}\leq \prod_{t=1}^{k-1}\alpha_t \cdot \prr{x_k\s{\beta}\in\mathbb{X}_k}\leq \dx, \label{eq:abc}
	\end{salign}
	if all control laws $u_k\s{i}$, for which $\Pf\s{i}\in\Pf\s{>\beta}$ holds, are truncated to $u_k\s{\beta}$.
	Therefore, the necessary condition results with \eqref{eq:tailoredlikelihoods} to:
	\begin{salign}
		\prr{x_k\s{\beta}\in\mathbb{X}_k}\leq \frac{\dx}{\CR{\prod_{t=1}^{k-1}\alpha_t}}=\gamma_k\s{x}.
	\end{salign}
	This is at least satisfied, if the $\gamma_k\s{x}$-confidence ellipsoid of the state $x_k\s{\beta}$ lies within the admissible state set, i.e.:
	\begin{salign}
		\mathcal{X}_k\s{\gamma,\beta}\subseteq \mathbb{X}_k. \label{eq:confidencewithinsAS}
	\end{salign}
	Following the same steps of the proof to \cite[Proposition~5]{WC20}, condition \eqref{eq:confidencewithinsAS} is satisfied if the LMI $\Lx_{k+1}\hs{j}\sgeq0$ given in \eqref{eq:LMIconst} is satisfied with the tailored likelihood $\gamma_k\s{x}$ for each half-space $j\in\{1,\dots,n_{\mathbb{X}_{k+1}}\}$ of $\mathbb{X}_{k+1}$, where $n_{\mathbb{X}_{k+1}}$ denotes the number of half-spaces.\\
	The reasoning for the input follows analogously\CR{, but with $\alpha_k$ instead of the product.}  \hfill $\qed$

%% file: Inputs/Acknowledgments.tex
Partial financial support by the German Research Foundation (DFG) within the research priority program SPP 1914: \textit{Cyberphysical Networking} is gratefully acknowledged.